\documentclass[final,12pt]{elsarticle}
\usepackage{xcolor,soul,cancel}
\usepackage[color]{showkeys}
\definecolor{refkey}{rgb}{0,1,1}
\definecolor{labelkey}{rgb}{1,0,0}
\usepackage{amssymb}
\usepackage{amsmath}
\usepackage{amsthm}
\usepackage{graphicx}
\usepackage{float}
\journal{LAMA}

\newtheorem{rem}{Remark}
\newtheorem{thm}{Theorem}
\newtheorem{lem}{Lemma}
\newtheorem{cor}{Corollary}
\newtheorem{prop}[thm]{Proposition}
\newtheorem{Ex}{Example}
\newcommand{\eq} [1] {\begin{equation}\label{#1}\quad}
\newcommand{\en} {\end{equation}}
\newcommand{\scal}[1]{\langle#1\rangle}
\newcommand{\norm}[1]{\left\Vert#1\right\Vert}
\newcommand{\abs}[1]{\left\vert#1\right\vert}

\newcommand{\C}{\mathbb C}

\newcommand{\R}{\mathbb R}
\newcommand{\T}{\mathbb T}
\newcommand{\diag}{\operatorname{diag}}

\renewcommand{\Re}{\operatorname{Re}}
\newcommand{\im}{\operatorname{Im}}

\newcommand{\re}{\operatorname{Re}}

\newcommand{\Span}{\operatorname{Span}}
\begin{document}

\begin{frontmatter}

\title{On the numerical range of some block matrices \\ with scalar diagonal blocks.\tnoteref{support}}


\author{Titas Geryba and Ilya M. Spitkovsky}
\address{Division of Science and Mathematics, New York  University Abu Dhabi (NYUAD), Saadiyat Island,
P.O. Box 129188 Abu Dhabi, United Arab Emirates}

\tnotetext[support]{The results are partially based on the Capstone project of the first named author [TG] under the supervision of the second named author [IMS]. This work was supported in part by Faculty Research funding from the Division of Science and Mathematics, New York University Abu Dhabi.
\\ \hspace*{.5cm} Email addresses: tg1404@nyu.edu [TG]; imspitkovsky@gmail.com, ims2@nyu.edu, ilya@math.wm.edu [IMS]}

\begin{abstract}
Several new verifiable conditions are established for matrices of the form $\left[\begin{smallmatrix} \alpha I_{n-k} & C \\ D & \beta I_k \end{smallmatrix}\right]$ to have the numerical range equal the convex hull of at most $k$ ellipses. For $k=2$, these conditions are also necessary, provided that the ellipses are co-centered.
\end{abstract}


\end{frontmatter}

\section{Introduction}

As usual, let $\C$ denote the field of complex numbers, $\C^{n\times m}$ stand for the set of all $n$-by-$m$ matrices with complex entries, with this notation abbreviated to just $\C^n$ when $m=1$. The standard
scalar product on $\C^n$ will be denoted $\scal{,.,}$, with the respective norm $\norm{.}$ defined by $\norm{x}=\scal{x,x}^{1/2}$.

The {\em numerical range} (a.k.a. the {\em field of values} or the {\em Hausdorff set}) of a matrix $A\in\C^{n\times n}$  is defined as
\[ W(A)=\{ \scal{Ax,x}\colon x\in\C^n,\ \norm{x}=1\}. \]
Introduced a century ago in the pioneering work by O.~Toeplitz \cite{Toe18} and Hausdorff \cite{Hau}, it has been researched extensively thereafter, see e.g. \cite{GusRa} or \cite[Chapter 1]{HJ2}.
Being the image of the unit sphere of $\C^n$ under a continuous mapping $f_A\colon x \mapsto \scal{Ax,x}$, the numerical range $W(A)$ is a closed, bounded, and connected subset of $\C$.
Moreover, $W(A)$ is convex (the classical Toeplitz-Hausdorff theorem). More specifically, it is the convex hull of the so called numerical range {\em generating} curve $C(A)$ having the following characteristic property: for every angle $\theta$ there are exactly $n$ tangent lines of $C(A)$, counting multiplicities, forming the angle $\theta$ with the positive direction of $x$-axis, and their intercepts with the orthogonal family of lines are exactly the eigenvalues of $\im(e^{-i\theta}A)$.
{
A more detailed description of the generating curve can be found in \cite{Ki}, where the notion was introduced, or its English translation \cite{Ki08}; see also  \cite{Fie81}.
}

Here and below for any square matrix $X$ \[ \re X=\frac{1}{2}(X+X^*) \text{ and } \im X =\frac{1}{2i}(X-X^*)
\] are the hermitian components from the representation \[X=\re X+i\im X. \]

In this paper we will consider exclusively matrices of the form
\eq{A} A=\begin{bmatrix} \alpha I_{n-k} & C \\ D & \beta I_k \end{bmatrix},\en
with their diagonal blocks being scalar multiples of the identity. Of course, for $n=2$ all matrices have form \eqref{A}, while $W(A)$ in this case is an elliptical disk with the foci at the eigenvalues of $A$, degenerating -- for normal $A$ -- into the line segment connecting the eigenvalues  (the Elliptical Range theorem). Our goal is to see the extent to which this shape of $W(A)$ persists for matrices \eqref{A} in higher dimensions.

More specifically, we are interested in the cases when the numerical range of $A$ is the convex hull of a finite number of ellipses.

A known general (but not very constructive) condition for this to occur is stated in Section \ref{s:pre}, along with other preliminary results. Several easily verifiable sufficient conditions, applicable to matrices \eqref{A}, are in Section~\ref{s:main}. Finally, if one of the diagonal blocks in \eqref{A} is of the size 2-by-2, the criterion is obtained in Section~\ref{s:k=2} for $W(A)$ to be the convex hull of at most two co-centered ellipses.

\section{Preliminaries}\label{s:pre}

Due to an elementary property
\[ W(\omega A+tI)=\omega W(A)+t \text{ for any } t,\omega\in\C, \]
it suffices to consider matrices \eqref{A} with $\alpha+\beta=0$. We will refrain from doing so in the statements, but will be taking advantage of this simplification in proofs, whenever convenient.  The quantity \eq{gamma} \gamma=\frac{\alpha-\beta}{2},\en invariant under shifts of $A$, will play an important role in what follows.

We will suppose, whenever convenient, that $k\leq n/2$. This can be done without loss of generality, since an appropriate permutational similarity can be used to switch from \eqref{A} to the matrix \eq{B} \begin{bmatrix} \beta I_{k} & D \\ C & \alpha I_{n-k} \end{bmatrix},\en leaving the numerical range unchanged.

Due to the construction of $C(A)$, the crucial role in the description of $W(A)$ for any matrix $A$ is played by the eigenvalues of $\im(e^{-i\theta}A)$. The lemma below provides the pertinent information for matrices of the form \eqref{A}.

\begin{lem}\label{l:eig}Let $A$ be of the form \eqref{A}, with $k\leq n/2$. Then the eigenvalues of $\im(e^{-i\theta}A)$ are $\im(e^{-i\theta}\alpha)$, of multiplicity $n-2k(\geq 0)$, with the remaining $2k$ given by the formula
	
\eq{la} \lambda_j=\frac{1}{2}\left( \im(e^{-i\theta}(\alpha+\beta))\pm\sqrt{ (\im(e^{-i\theta}(\alpha-\beta))^2+\mu_j(\theta)}\right), \quad j=1,\ldots,k.\en
Here $\mu_j(\theta)$ are the eigenvalues of the matrix \eq{M} M(\theta):=H-2\re(e^{-2i\theta}Z),\en
with $H$ and $Z$ defined by
\eq{HCD} H=C^*C+DD^*,\  Z=DC. \en
\end{lem}
\begin{proof}As was discussed, it suffices to consider the case $\beta=-\alpha$, in which formula \eqref{la} simplifies
to 	
\eq{la0} \lambda_j=\pm\sqrt{(\im(e^{-i\theta}\alpha))^2+\mu_j(\theta)/4}, \quad j=1,\ldots,k.\en
The result then follows from the Schur complement formula for the determinant, according to which the characteristic polynomial  of $\im(e^{-i\theta}A)$ is \begin{multline*} \det \left((\im(e^{-i\theta}\alpha)-\lambda)I_{n-k}\right)\\
\cdot \det \left((-\im(e^{-i\theta}\alpha)-\lambda)I_k-(\im(e^{-i\theta}\alpha)-\lambda)^{-1}M(\theta)/4\right) \\
= (\im(e^{-i\theta}\alpha)-\lambda)^{n-2k}  \det \left((\lambda^2-(\im(e^{-i\theta}\alpha))^2)I_k-M(\theta)/4\right).\end{multline*} \end{proof}
The eigenvalues $\im(e^{-i\theta}\alpha)$ of $\im(e^{-i\theta}A)$, if actually present (which is the case when $k<n/2$), correspond to the tangent lines of the numerical range generating curve $C(A)$ all passing through $\alpha$. This means simply that in this case the singleton $\{\alpha\}$ is a component of $C(A)$. For $k>n/2$ we just need to replace $\alpha$ with $\beta$ in this reasoning.

The remaining tangent lines form a family which is central symmetric with respect to $(\alpha+\beta)/2$, as \eqref{la} implies. Observe also that $\alpha,\beta\in W(A)$, independent of whether or not $\alpha\in C(A)$. We thus arrive at

\begin{prop}Let $A$ be of the form \eqref{A}. Then $C(A)\setminus\{\alpha,\beta\}$, and thus the numerical range $W(A)$ itself, is central symmetric with respect to $(\alpha+\beta)/2$.  \end{prop}
Note that for $\alpha=\beta\ (=0)$ this property follows directly from the equality $-A=JAJ$, where $J=\diag[I_{n-k}, -I_k]$, as observed in \cite[Proposition 3.2]{Yeh}.

So, if $W(A)$ is the convex hull of finitely many ellipses, some of them may be centered at $(\alpha+\beta)/2$, with the other coming in pairs, central symmetric with respect to $(\alpha+\beta)/2$, but centered elsewhere.

The following example shows that such pairs can indeed be present.

\begin{Ex} \label{ex:yeh}
 Let in \eqref{A} $\alpha = \beta = 0$, $C = B^* + I$ and $D = -I + B$, where $B$ is a $2$-by-$2$ matrix of the form
    \begin{align*}
        B = \begin{bmatrix}
            a & \sqrt{(1-|a|^2)(1-|c|^2)} \\ 0 & c
        \end{bmatrix},
    \end{align*}
$\abs{a}=\abs{c}\leq 1$ and $\im a = \im c$. Then, as shown in {\em \cite[Theorem 3.14]{Yeh}}, the numerical range of $A$ is a convex hull of two ellipses that are centered away from the origin.
    \begin{figure}[H]
    \centering
    \includegraphics[scale=0.4]{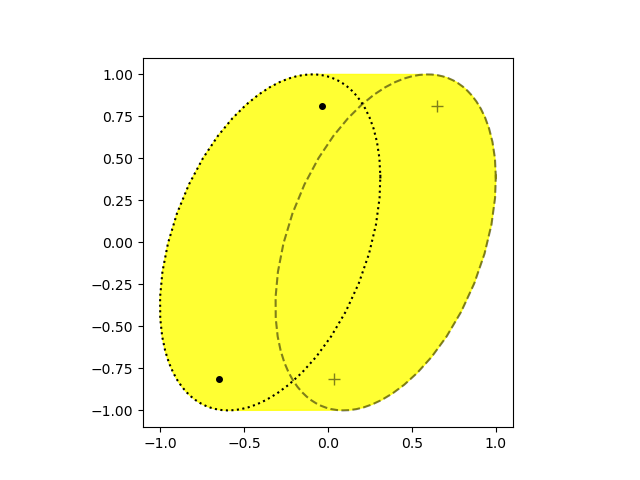}
    \includegraphics[scale=0.4]{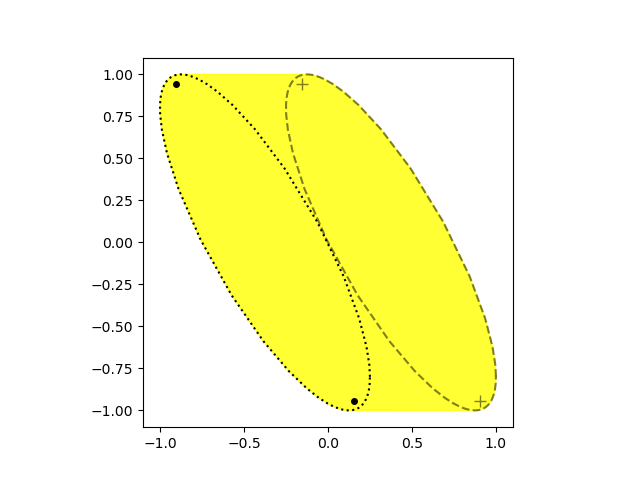}
    \caption{$W(A)$ for example~\ref{ex:yeh}. On the left $a =  c = \frac{1}{2} + \frac{i}{4}$. On the right $a = c  = -\frac{i}{2}$.}
    \label{fig:yehEx}
    \end{figure}
\end{Ex}

In this paper, however, we will restrict our attention to the case when such pairs do not materialize, that is, all the ellipses forming $C(A)$ are co-centered.
The following technical lemma will therefore be instrumental.

\begin{lem}\label{l:cel} Given a matrix $A\in\C^{n\times n}$, its numerical range $W(A)$ is the convex hull of
$m$ ellipses centered at the origin if and only if for some partition {$\{\gamma_j\}_{j=1}^m$} of the unit circle $\T$, the maximal eigenvalue $\lambda_{\max}(\theta)$ of $\im(e^{-i\theta}A)$ satisfies
\eq{lam} \lambda^2_{\max}(\theta)=A_j\cos(2\theta)+B_j\sin(2\theta)+C_j, \quad e^{i\theta}\in\gamma_j,\en
for some { distinct triples} $\{A_j,B_j,C_j\}$ { of positive numbers such that}  $C_j>\sqrt{A_j^2+B_j^2}$, $j=1,\ldots,m$.
\end{lem}\label{l:ell}
{ Before passing to the proof of Lemma~\ref{l:cel}, let us clarify that for $m=1$ the only element of the partition is the whole circle $\T$, while for $m>1$ each $\gamma_j$ is a pair of central symmetric circular arcs, due to the central symmetry of all the ellipses involved.}
{
\begin{proof}
 From the description of $C(A)$ it easily follows (and also is well known)  that the supporting line of $W(A)$ in the direction of $e^{i\theta}$ is passing through the point $e^{i(\theta+\pi/2)}\lambda_{\max}\left(\im(e^{-i\theta}A)\right)$. A direct computation (performed, e.g., in the proof of \cite[Theorem 1]{ChienHung}, where the case $m=1$ was treated)
shows that the boundary of $W(A)$ contains an arc $E$ of the ellipse with the half-axes $\sqrt{C\pm\sqrt{A^2+B^2}}$
and the major axis forming angle $-\frac{1}{2}\tan^{-1}\frac{B}{A}$ with the positive direction of the $x$-axis if and only if
\[ \lambda^2_{\max}(\theta)=A\cos(2\theta)+B\sin(2\theta)+C\]
for all the angles $\theta$ formed by the tangent lines of the elliptical arc $E$ with the positive direction of the $x$-axis.
Condition \eqref{lam} therefore holds if and only if the boundary of $W(A)$ consists of elliptical arcs of exactly $m$ ellipses, perhaps connected by line segments.
\end{proof}
}

\section{Main results} \label{s:main}

 We will now consider several cases in which the eigenvalues of  the matrices $M(\theta)$  given by \eqref{M} can be computed explicitly, thus yielding a full description of $C(A)$ and, ultimately, of $W(A)$.  All the results will be stated in terms of matrices $H,Z$ defined by \eqref{HCD}.

 One such case materializes if the matrix $Z$ is normal and commutes with $H$. These matrices can be then diagonalized by the same (unitary) similarity. Let us label the eigenvalues $h_j,z_j$ of $H$ and $Z$, according to the order in which they appear on the respective diagonals.

 \begin{thm}\label{th:tri}Let the off-diagonal blocks of \eqref{A} be such that in \eqref{HCD} the matrix $Z$ is normal and commutes with $H$. Then $W(A)$ is the convex hull of at most $k$ ellipses.

Moreover, each ellipse associated with a particular $\mu_j(\theta)$ is centered at $(\alpha + \beta)/2$ with a major axis parallel to $e^{-i \phi_j}$ of length $2a_j$ and minor axis of $2b_j$ which are given by
\begin{equation*}
    a_j = \left( \frac{|\gamma|^2}{2} + \frac{h_j}{4} + \frac{1}{2}|
    \gamma^2 + z_j| \right)^{\frac{1}{2}}, \quad  b_j =  \left( \frac{|\gamma|^2}{2} + \frac{h_j}{4} - \frac{1}{2}|\gamma^2 + z_j| \right)^{\frac{1}{2}}.
\end{equation*}
Here $\phi_j$ is the principal argument of $i\sqrt{\gamma^2 + z_j}$.

\end{thm}
\begin{proof}
Under the conditions imposed on $H$ and $Z$, the matrix \eqref{M} is diagonalizable for all values of $\theta$ by the same similarity as $Z$ and $H$, implying that
\eq{muj}
        \mu_j(\theta) = h_j - 2 \re (e^{-2i\theta} z_j) = h_j - 2\re(z_j)\cos 2\theta - 2\im (z_j) \sin 2\theta.
\en
According to Lemma~\ref{l:eig} the equation of $\lambda_j(\theta)$ corresponding to $A-((\alpha + \beta)/2)I$ is given by:
\begin{align*}
    \lambda_j^2
    &
    =
    \im \left( e^{-i\theta} \gamma  \right)^2 + \mu_j(\theta)/4
    \\
    &
    =
    \im(\gamma)^2 \cos^2 \theta + \re(\gamma)^2 \sin^2 \theta  - \re(\gamma)\im(\gamma)\sin 2 \theta + \mu_j(\theta)/4
    \\
    &
    =
    \frac{|\gamma|^2}{2} + \frac{h_j}{4} - \frac{\re(z_j + \gamma^2)}{2} \cos 2 \theta
    - \frac{\im(z_j + \gamma^2)}{2} \sin 2 \theta.
\end{align*}
By Lemma~\ref{l:ell} the above corresponds to an ellipse centered at the origin with minor and major axes as described in the statement. Translating these ellipses by $(\alpha + \beta)/2$ we obtain the result.
\end{proof}
\begin{rem} Proof of Theorem~\ref{th:tri} goes through if we merely suppose that the matrices $Z,Z^*$ and $H$ can
be put in a triangular form by a simultaneous similarity. However, this (formally weaker) requirement in fact implies that $Z$ is normal and commutes with $H$. \end{rem}

Conditions of Theorem~\ref{th:tri} hold in particular when $Z$ is a scalar multiple of the identity matrix. Then all $z_j$ are the same, implying that the ellipses described by Theorem~\ref{th:tri} form a nested family. Therefore, independent of the value of $k$, $W(A)$ is an elliptical disk. More specifically:
\begin{cor}\label{co:DC}Let the off-diagonal blocks of \eqref{A} be such that $DC=cI$. Then $W(A)$ is
an elliptical disk centered at $(\alpha +\beta)/2$ with  the axes of length
$\sqrt{\norm{H}+2\abs{\gamma}^2\pm 2\abs{\gamma^2 +c}}$, and the major axis parallel to $e^{-i \phi}$, where \iffalse of length  $\sqrt{\mu +2\abs{\gamma}^2 + 2\abs{\gamma^2 + c^2}}$ and minor axis of length $\sqrt{\mu +2\abs{\gamma}^2 - 2\abs{\gamma^2 + c^2}}}$. Here $\mu=\norm{H}$, $\gamma = (\alpha-\beta)/2$, and\fi  $\phi$ is the (principal) argument of $i\sqrt{\gamma^2 + c}$.
\end{cor}
The description of $W(A)$ simplifies even further when $c=0$, i.e., $DC=0$. The axes of the ellipse $W(A)$ then have length $\sqrt{\norm{H}^2+4\abs{\gamma}^2}$ and $\norm{H}$.

Condition $DC=0$ in geometrical terms may be recast as the kernel of $D$ containing the range of $C$. This happens, in particular, when the $j$-th row of $C$ consists of all zeros provided that the $j$-th column of $D$ is non-zero. This covers Scenario~2 from \cite[Theorem 7]{ChienHung}, corresponding to $j=1$.
	
Of course, $DC=0$ when $D=0$. The latter case is listed as Scenario~1 in \cite[Theorem 7]{ChienHung}. Note however that all quadratic matrices (i.e., matrices with minimal polynomial of degree 2) can be reduced to this case via a unitary similarity --- the approach used in \cite{TsoWu} when proving the ellipticity of the numerical ranges of such matrices.

Another setting obviously guaranteeing $DC=cI$ is when $k=1$. The ellipticity of $W(A)$ in this case was proved in \cite{Linden}, see also \cite[Theorem 2]{ChienHung}.

Yet another situation in which conditions of Theorem~\ref{th:tri} hold is when there exist unitary $U\in\C^{k\times k}$ and $V\in\C^{(n-k)\times(n-k)}$ such that $UDV$ and $UC^*V$ are both diagonal. Indeed, then matrices \eqref{HCD} are simultaneously diagonalizable under the unitary similarity $X\mapsto UXU^*$. This case, along with a number of situations reducing to it, was extensively treated in \cite{BS04}. It was also observed in \cite{BS04} that the existence of such $U,V$ is equivalent to normality of two matrices, $CD$ and $DC$, see, e.g., \cite[p. 426]{HJ1}. The normality of $DC$ of course holds when $k=1$ while there is no reason for $CD$ to also be normal.

Situations with $Z=DC$ being normal and commuting with $H$, while $CD$ is not normal, do materialize in more interesting settings as well. Consider for example the case $n=4, k=2$,
\[ D=\begin{bmatrix}\xi & \eta \\ \overline{\eta} & \zeta\end{bmatrix}U \text{ and } C=D^{-1}\begin{bmatrix}\omega & 0\\ 0 & \overline{\omega}^{-1}\end{bmatrix}, \]	
with $U\in\C^{2\times 2}$ unitary, $\xi,\zeta>0$, $\xi\zeta-\abs{\eta}^2=1$, $\xi\neq\zeta$, and $\abs{\omega}\neq 1$.

Then \[ Z=DC=\begin{bmatrix}\omega & 0\\ 0 & \overline{\omega}^{-1}\end{bmatrix} \]
is normal, while a direct computation shows that $H= Z^*(DD^*)^{-1}Z+DD^*$ is diagonal, and therefore commutes with $Z$.
At the same time $CD=D^{-1}ZD$ is unitarily similar to
\[ \begin{bmatrix}\zeta & -\eta \\ -\overline{\eta} & \xi\end{bmatrix}Z\begin{bmatrix}\xi & \eta \\ \overline{\eta} & \zeta\end{bmatrix}=\begin{bmatrix} * & \eta\zeta(\omega-\overline{\omega}^{-1}) \\  -\overline{\eta}\xi(\omega-\overline{\omega}^{-1}) & * \end{bmatrix}. \]
The off-diagonal entries of the resulting matrix differ in magnitude, precluding it from being normal. Consequently, the matrix $CD$ is not normal either.

Here is a particular numerical example, to illustrate.

\begin{Ex} \label{ex:HZcom}
    Consider $C$ and $D$ as follows:
    \begin{align*}
        C = \frac{1}{2}
        \begin{bmatrix}
        8  & -1 \\ -4 & 1
        \end{bmatrix}
        \quad
        D =
        \begin{bmatrix}
        1  & 1 \\ 1 & 2
        \end{bmatrix}
    \end{align*}
Then
   \[
DC = \frac{1}{2}\begin{bmatrix} 4  & 0 \\ 0 & 1 \end{bmatrix}   \text{ is normal, while for }
CD = \frac{1}{2} \begin{bmatrix} 7 & 6 \\ -3 & -2 \end{bmatrix}
  \]
 we have
\[ (CD)^*CD - CD(CD)^* = \frac{1}{4} \begin{bmatrix} -27  & 81 \\ 81 & 27\end{bmatrix} \neq 0.  \]
 Finally, for this choice of $C,D$: \[
H= \frac{1}{4}\begin{bmatrix}8 & -4 \\ -1 & 1\end{bmatrix}\begin{bmatrix}8 & -1 \\ -4 & 1\end{bmatrix}+
\begin{bmatrix}1 & 1 \\ 1 & 2\end{bmatrix}^2=\begin{bmatrix}22 & 0\\ 0 & 5.5\end{bmatrix}.	\]
is indeed diagonal.

    \begin{figure}[H]
    \centering
    \includegraphics[scale=0.4]{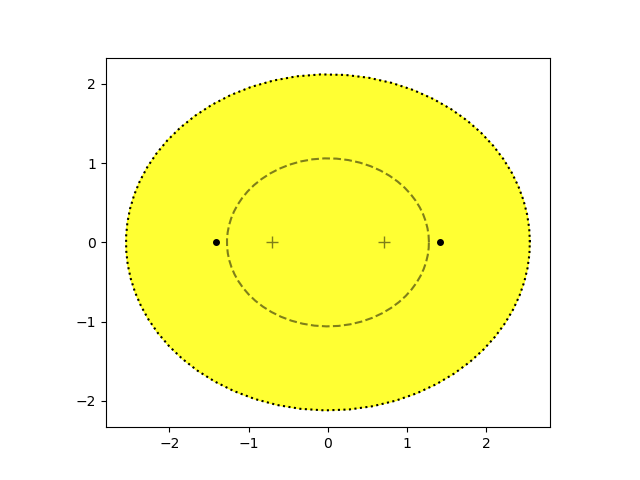}
    \includegraphics[scale=0.4]{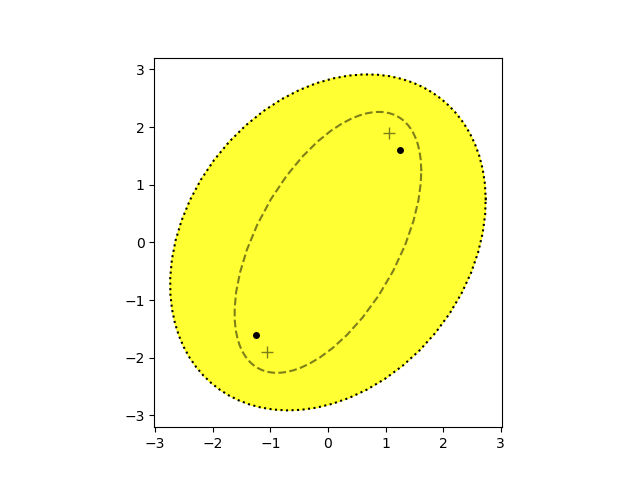}
    \caption{$W(A)$ from Example~\ref{ex:HZcom}; $\alpha = \beta = 0$ on the left and $\alpha = -\beta = 1 + 2i$ on the right. }
    \label{fig:HZcomEll}
    \end{figure}
\end{Ex}


As can be seen from Corollary~\ref{co:DC}, in the setting of Theorem~\ref{th:tri} the number $m(A)$ of ellipses actually involved in generating $W(A)$ may be strictly smaller than its upper bound $k$ (unless, of course, $k=1$). A moment's thought reveals that $m(A)$ is the number of distinct functions \eqref{muj} attaining the maximal (with respect to $j$) value while $\theta$ varies in $[0,\pi]$.

Determining the value of $m(A)$ in the setting of Theorem~\ref{th:tri} without imposing additional restrictions on the matrices \eqref{HCD} may be rather challenging. Here is yet another case in which it can be done. Recall that a matrix $X$ is {\em essentially Hermitian} if it is normal with a spectrum lying on a line (equivalently: if $W(X)$ has empty interior).

\begin{cor}\label{co:tri}
Let $Z$ be essentially Hermitian.  Suppose in addition that $H$ is a scalar multiple of the identity, while $Z$ is not.
Then $W(A)$ is the convex hull of two non-nested ellipses, both centered at $(\alpha + \beta)/2$ with their major (and thus minor) axes forming the right angle.
\end{cor}
\begin{proof}
Under the conditions imposed on $H$ and $Z$, from \eqref{muj} it follows that \[ \mu_s(\theta)-\mu_t(\theta)= 2\re(e^{-2i\theta}(z_t-z_s)), \] with the arguments of all the non-zero differences $z_t-z_s$ being equal $\mod\pi$.
So, the switch in the sign of $\mu_s(\theta)-\mu_t(\theta)$ for all $s,t=1,\ldots,k$ such that $z_t\neq z_s$ does actually occur. Moreover, it happens at the same two values of $\theta$, different by $\pi/2$.
\end{proof}
In particular, the above corollary holds when $Z$ is Hermitian or Skew-Hermitian. The following is such an example. Consider:
\begin{align*}
A = \begin{bmatrix} 0 & -I_k + B \\ I_k + B^* & 0 \end{bmatrix}
\end{align*}
where $B$ is unitary. In this case, we obtain that $H = 4I_k$ and $Z = B^*-B$. Since $H$ is a scalar multiple of the identity, and $Z$ is skew-Hermitian, we conclude that $W(A)$ is the convex hull of two ellipses.

 The exactly same reasoning applies to the case of $A$ given by
	\begin{align*}
	A = \begin{bmatrix} 0 & -I'_k + B \\ I'_k + B^* & 0 \end{bmatrix},
	\end{align*}
where $I_k'$ is the identity $k$-by-$k$ matrix with the last column deleted, and $B\in\C^{k,k-1}$ is an isometry. This case was considered in \cite[Theorem~3.5]{Yeh}, see also \cite[Theorem 3.4]{CGWWY}.

We now turn to the situation when the eigenvalues $\mu_j(\theta)$ do not depend on $\theta$ in spite of the fact that the matrices $M(\theta)$ themselves do.

\begin{thm}\label{th:nilp} In the notation \eqref{HCD}, suppose that, for some subspace $L$ invariant under $H$, the range of $Z$ is contained in $L$ which in turn is contained in the kernel of $Z$. Then
$W(A)$ is an elliptical disk centered at $(\alpha + \beta)/2$ with major axis parallel to $e^{-i\phi}$ of length  $\sqrt{\mu  +4\abs{\gamma}^2 }$ and  minor axis of length $\sqrt{\mu}$, where $\mu$ is the largest eigenvalue of $H-2Re Z$,  $\gamma = (\alpha-\beta)/2$ and $\phi$ is the (principal) argument of $i \gamma$.
\end{thm}

\begin{proof}
With respect to the decomposition $\C^n= L\oplus L^\perp$, $H$ and $Z$ can be expressed as:
\begin{equation*}
    H = \begin{bmatrix} H_1 & 0 \\ 0 & H_2 \end{bmatrix},  \quad Z = \begin{bmatrix} 0 & Z_0 \\ 0 & 0 \end{bmatrix},
\end{equation*}
and therefore \[ M(\theta)= H - 2\re(e^{-2i \theta} Z)=\begin{bmatrix} H_1 & -e^{-2i\theta}Z_0 \\  -e^{2i\theta}Z_0^* & H_2\end{bmatrix}=U^*MU, \] where $U$ is the unitary matrix $\diag[I_{n-k},-e^{2i\theta}I_k]$,
and the matrix \[ M=\begin{bmatrix} H_1 & Z_0 \\ Z_0^* & H_2 \end{bmatrix} \] is independent of $\theta$.
Consequently, the eigenvalues $\mu_j$ of $M(\theta)$ are the same as those of $M$, and thus also do not depend on $\theta$.

By Lemma~\ref{l:eig} applied to $A_0=A-\frac{\alpha+\beta}{2}I$, the respective eigenvalues $\lambda_j(\theta)$ satisfy
\begin{align*}
    \lambda_j^2
    &
    =
     \left(\im (e^{-i\theta} \gamma)\right)^2 + \mu_j/4
    \\
    &
    =
   (\im\gamma)^2 \cos^2 \theta + (\re\gamma)^2 \sin^2 \theta  - \re\gamma\im\gamma\sin 2 \theta + \mu_j/4
   \\
   &
   =
    \frac{\mu_j}{4} + \frac{|\gamma|^2}{2} - \frac{\Re(\gamma^2)}{2} \cos 2 \theta - \frac{\im(\gamma^2)}{2}\sin 2 \theta.
\end{align*}
Consequently, in the notation of Lemma~\ref{l:ell}, \[ \lambda_{\max}^2(\theta)=\frac{\mu}{4} + \frac{|\gamma|^2}{2} - \frac{\Re(\gamma^2)}{2} \cos 2 \theta - \frac{\im(\gamma^2)}{2}\sin 2 \theta, \]
implying (by the same Lemma) that $W(A_0)$ is the elliptical disk centered at the origin with major axis parallel to $e^{-i \phi}$ of length $\sqrt{\mu  + 4|\gamma|^2 }$ and  minor axis of length $\sqrt{\mu}$.  Translating the matrix $A_0$ by $(\alpha + \beta)/2$ we obtain the result.
\end{proof}

Conditions imposed on $Z$ by Theorem~\ref{th:nilp} imply that it is nilpotent, namely $Z^2=0$.  If $H$ is a scalar multiple of the identity, then this is also sufficient.   Another sufficient condition is $CD=0$. Indeed, setting $L=\ker C$ and making use of the fact that the range of $D$ lies in $L$,  we then obtain the following matrix representations:
\[ C= \begin{bmatrix} 0 & C_{12} \\ 0 & C_{22}\end{bmatrix}, \quad D=\begin{bmatrix}D_{11} & D_{12} \\ 0 & 0\end{bmatrix} \] with respect to the decomposition $L\oplus L^\perp$.
 Direct computations show then
\[ Z=\begin{bmatrix} 0 & D_{11}C_{12}+D_{12} C_{21} \\ 0 & 0\end{bmatrix}, \quad
H=\begin{bmatrix}D_{11}D_{11}^*+D_{12}D_{12}^* & 0 \\ 0 & C_{12}^*C_{12}+C_{22}^*C_{22}\end{bmatrix},
\]
	implying that indeed so chosen $L$ is invariant under $H$ and contains the range of $Z$.

 This case covers the remaining Scenario~3 of \cite[Theorem 7]{ChienHung} in which $C$ has the first zero column while $D$ is a matrix with all rows being zero except maybe the first one.

Note, however, that the same conclusion can be reached by passing from \eqref{A} to \eqref{B} and invoking the result for the case $DC=0$ treated earlier.

 In case of square matrices $C,D$ (i.e., when $k=n/2$), the nilpotency of $Z$ implies that $C$ and $D$ cannot both be invertible. The following class of examples shows that one of them still can.

Let $k=n/2$, choose $C$ having non-zero pair-wise orthogonal columns, and let $D=JC^{-1}$, where $J$ has the block structure $\begin{bmatrix} 0 & T \\ 0 & 0\end{bmatrix}$ with an arbitrary $k$-by-$k$ block $T$.
Then of course $Z=J$, the matrix $C^*C$ is diagonal and invertible, implying the block diagonal structure of $H=C^*C+J(C^*C)^{-1}J^*$. So, conditions of Theorem~\ref{th:nilp} are satisfied by $L$ being the span of the first $k$ vectors $e_1,\ldots,e_k$ of the standard basis. A numerical example is below.

\begin{Ex} \label{ex:nilp2}
Consider $C$ and $D$ as follows:
\begin{align*}
    C=\begin{bmatrix}
    1+i & 1+i \\
    1 &  -2 \end{bmatrix},
    \qquad
    D =\frac{1}{6}\begin{bmatrix}1 - i & -2\\0 & 0\end{bmatrix}.
\end{align*}
We have then:
\begin{align*}
    Z =
    \begin{bmatrix}
    0 & 1 \\
    0 & 0
    \end{bmatrix},
    \qquad
    H = \frac{1}{6}\begin{bmatrix}
        19 & 0 \\
        0 & 36
    \end{bmatrix},
\end{align*}
and $L=\Span\{e_1\}$ does the job.

\begin{figure}[H]
    \centering
    \includegraphics[scale=0.4]{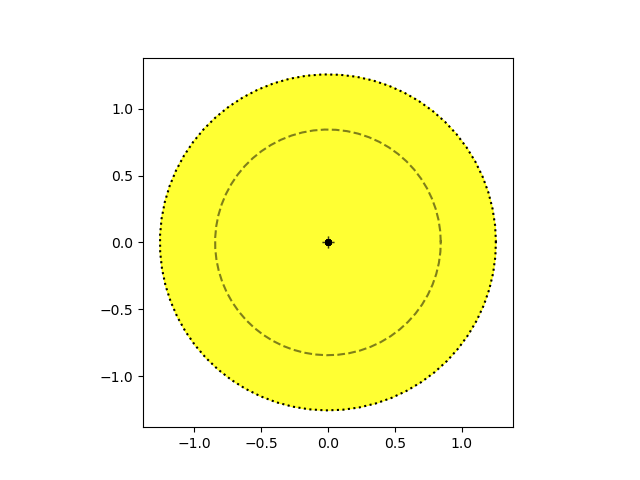}
    \includegraphics[scale=0.4]{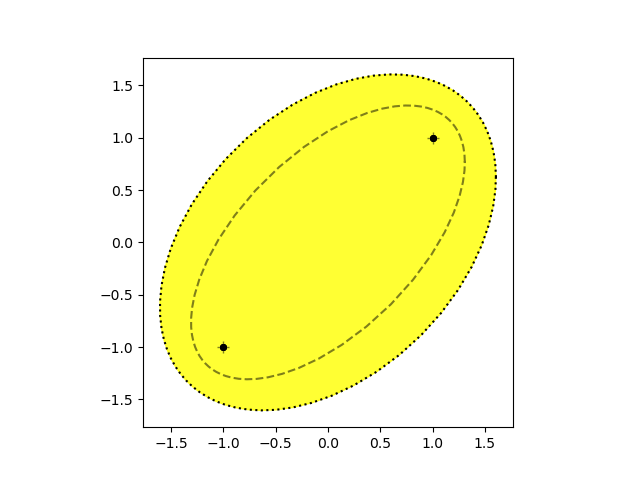}
    \caption{$W(A)$ from example~\ref{ex:nilp2}; $\alpha = \beta = 0$ on the left and $\alpha = -\beta = 1+i$ on the right.}
    \label{fig:ellnil}
\end{figure}
\end{Ex}

\section{The $k=2$ case}\label{s:k=2}

In this section, we concentrate on the case $k=2$. Independent of the value of $n\ (\geq 4)$, matrices \eqref{HCD} are then of 2-by-2 size, and the eigenvalues $\mu_j(\theta)$ of \eqref{M} can be found explicitly by solving  a quadratic equation. Namely,
denoting by $h_j$ and $z_j$ the eigenvalues of $H$ and $Z$, respectively, and by $z_{ij}$ the entries of $Z$ in the orthonormal basis of the eigenvectors of $H$ {(by convention choosing $Z$ to be diagonal when $H$ is a scalar multiple of the identity)}:

\eq{mus}
\mu_{1,2} = \frac{1}{2} \left( h_{1} + h_{2} - 2\Re(z_{11} + z_{22})\cos 2 \theta - 2\im(z_{11} + z_{22})\sin 2 \theta
\pm
\sqrt{\Delta} \right)
\en
where $\Delta$ is given by:
\eq{D}\begin{aligned}
\Delta =
&
(h_{1} - h_{2})^2 + 2\abs{z_{22}-z_{11}}^2 +  4|z_{12}|^2 + 4|z_{21}|^2
\\
+
&
4 (h_{1} - h_{2})\Re(z_{22} - z_{11}) \cos 2 \theta
\\
+
&
4 (h_{1} - h_{2})  \im(z_{22} - z_{11}) \sin 2\theta
\\
+
&
2 \Re( (z_1 - z_2)^2 ) \cos 4\theta
\\
+
&
2 \im( (z_1 - z_2)^2 ) \sin 4\theta
\end{aligned}\en

\begin{thm}\label{th:k=2} Let in \eqref{A} $k=2$. Then $W(A)$ is the convex hull of ellipses centered at $(\alpha+\beta)/2$ if and only if at least one of the following conditions holds: \\ {\em(i)} $Z$
is normal commuting with $H$,\\ {\em (ii)} eigenvalues of $Z$ coincide, and it shares an eigenvector with $H$, \\ {\em (iii)} $H$ and $Z$ are such that
\eq{iv} -(h_1-h_2)^2\frac{z_{12}z_{21}}{(z_1-z_2)^2}=(\abs{z_{12}}+\abs{z_{21}})^2.\en

\end{thm}
{ Observe that conditions (i) and (ii) hold simultaneously if and only if $Z$ is a scalar multiple of the identity, cases (ii) and (iii) are mutually exclusive, and ``(i) but not (ii)'' is a subcase of (iii). Examples \ref{ex:nilp2} and  \ref{ex:HZcom} show, respectively, that situations ``(ii) but not (i)'' and ``(i) but not (ii)'' indeed materialize. As for ``(iii) but not (i)'', the respective examples also exist. Their construction is more involved, however, and can be achieved by solving certain non-linear matrix equations with the use of \cite{ERR93}. This issue will be addressed elsewhere.}
\begin{proof} According to \eqref{mus} and \eqref{la}, $\lambda_1^2(\theta)+\lambda_2^2(\theta)$ is a real linear combination of $\cos 2\theta, \sin 2\theta$ and 1. Consequently, each of $\lambda_j^2$ is such a linear combination if and only if their difference is. Equivalently, this happens if and only if $\Delta$ has the form $(a\cos 2\theta +b\sin 2\theta +c)^2$
for some $a,b,c\in\R$.

As it can be seen from \eqref{D}, this in its turn is equivalent to the system
\eq{abc}\begin{aligned} a+ib=\pm{2 (z_1-z_2)}, \\
 (a+ib)c=2(h_1-h_2)(z_{22}-z_{11}), \\
c^2=(h_1-h_2)^2+2\abs{z_{22}-z_{11}}^2+4\abs{z_{12}}^2+4\abs{z_{21}}^2-2\abs{z_1-z_2}^2.\end{aligned} \en
It therefore remains to show that one of the conditions (i)--(iii) holds if and only if the system \eqref{abc} is consistent. To this end, consider several possibilities separately, keeping in mind that
\eq{z12} (z_1-z_2)^2=(z_{11}-z_{22})^2+4z_{12}z_{21}. \en

{\sl Case 1.} $z_1=z_2$, $h_1=h_2$, implying immediately  that (ii) holds.  On the other hand, the solution of \eqref{abc} is delivered by
\[ a=b=0,\quad  c=\pm (2\abs{z_{22}-z_{11}}^2+4\abs{z_{12}}^2+4\abs{z_{21}}^2)^{1/2}. \]

{\sl Case 2.} $z_1=z_2$, $z_{11}=z_{22}$. According to \eqref{z12}, this is only possible when $z_{12}z_{21}=0$, implying that $H$ and $Z$ share a common eigenvector. So, (ii) holds again. And the system \eqref{abc} is consistent, having
\[ a=b=0, \quad c=\pm ((h_1-h_2)^2+4\abs{z_{12}}^2+4\abs{z_{21}}^2)^{1/2}\]
as its solution.

{\sl Case 3.} $z_1=z_2$, while $h_1\neq h_2, z_{11}\neq z_{22}$. Then $Z$ is not normal, because the only normal matrices with coinciding eigenvalues are scalar multiples of the identity. Also, \eqref{z12} implies that $z_{12}z_{21}\neq 0$, and therefore $H$ and $Z$ do not have eigenvectors in common. Consequently, neither of (i)--(ii) holds.  And (iii) does not hold either, since the left hand side of \eqref{iv} is undefined.

On the other hand, the system \eqref{abc} is inconsistent. Indeed, the first equation implies that $a=b=0$, which is in contradiction with the second equation. 

{\sl Case 4.} $z_1\neq z_2$. The solution of \eqref{abc}, if exists, is given by
\eq{abc4} a = \pm2 \re (z_1-z_2), \ b = \pm 2\im (z_1-z_2), \ c=\pm\frac{(h_1-h_2)(z_{22}-z_{11})}{z_1-z_2} \en
(with the coordinated choice of the three signs), as follows from the first two equations.

{\sl Subcase 4A.} $h_1=h_2$ or $z_{11}=z_{22}$. Then $c=0$, and in order for \eqref{abc4} to indeed deliver a solution, according to the third equation it is necessary and sufficient that

\eq{hz} (h_1 - h_2)^2 + 2 \abs{z_{22}-z_{11}}^2+4\abs{z_{12}}^2+4\abs{z_{21}}^2-2\abs{z_1-z_2}^2=0. \en
If $z_{11} = z_{22}$, formula \eqref{z12} allows to substitute $\abs{z_1-z_2}^2$ by $4\abs{z_{12}z_{21}}$, and so \eqref{hz} can be rewritten as
\[ (h_1 - h_2)^2+4(|z_{12}| - |z_{21}| )^2=0. \]
So, condition $z_{11}=z_{22}$ implies $h_1 = h_2$, and we may concentrate on the latter.

According to \eqref{hz} then:

\[ 2 \abs{z_{22}-z_{11}}^2+4\abs{z_{12}}^2+4\abs{z_{21}}^2=2\abs{z_1-z_2}^2. \]

Comparing with \eqref{z12} we see that this happens if and only if
\[ \abs{z_{12}}=\abs{z_{21}}, \quad 2\arg(z_{11}-z_{22})=\arg z_{12}+\arg z_{21} \mod 2\pi \]
 (the condition on the arguments being imposed only when it makes sense, i.e., $z_{11}\neq z_{22},\ z_{12}z_{21}\neq 0$).

But this is exactly the normality criterion for $Z$.
 Consequently, (i) holds. Observe that $Z$ is not a scalar multiple of the identity since $z_1\neq z_2$.

{\sl Subcase 4B.} $h_1\neq h_2$. Then, as shown above, $z_{11}\neq z_{22}$, and in order for $c$ from \eqref{abc4} to be real, it is necessary and sufficient that \eq{argz} \arg (z_{11}-z_{22})=\arg (z_{1}-z_{2}) \mod \pi. \en Also, plugging in this value of $c$ into the third equation of \eqref{abc} yields
\[  \frac{(h_1-h_2)^2(z_{22}-z_{11})^2}{(z_1-z_2)^2} = (h_1-h_2)^2+2\abs{z_{22}-z_{11}}^2+4\abs{z_{12}}^2+4\abs{z_{21}}^2-2\abs{z_1-z_2}^2. \]
Taking \eqref{z12} into consideration, this simplifies to

\eq{iv-i} -2(h_1-h_2)^2\frac{z_{12}z_{21}}{(z_1-z_2)^2}=\abs{(z_1-z_2)^2-4z_{12}z_{21}}- \abs{z_1-z_2}^2+2\abs{z_{12}}^2+2\abs{z_{21}}^2. \en
Since the right hand side of \eqref{iv-i} is non-negative, this equality hold only if \eq{nonpos} \frac{z_{12}z_{21}}{(z_1-z_2)^2}\leq 0.\en

The latter inequality, when combined with \eqref{z12}, yields
\eq{z121} \abs{z_{11}-z_{22}}^2=\abs{z_1-z_2}^2+4\abs{z_{12}z_{21}}, \en
which allows to simplify \eqref{iv-i} further to \eqref{iv}.

On the other hand, \eqref{iv} implies \eqref{nonpos} (this is clear if $h_1\neq h_2$, and can be arranged by putting $Z$ in a triangular form if this is not the case). Along with \eqref{z12}, this guarantees \eqref{argz}. Thus, \eqref{abc} admits a real solution.

We have exhausted all possible cases, thus completing the proof. \end{proof}

Since $k=2$, for matrices $A$ satisfying conditions of Theorem~\ref{th:k=2} their generating curve $C(A)$ consists of two ellipses. Depending
on whether or not these ellipses are nested, the numerical range $W(A)$ itself will either be an elliptical disk, or the convex hull of two non-nested ellipses with the same center, thus having
four flat portions on its boundary. The next statement describes when which of the two possibilities occurs.

\begin{thm}\label{th:1or2}Let the matrix $A$ satisfy conditions of  Theorem~\ref{th:k=2}. Then $W(A)$ is the convex hull of two non-nested ellipses in case {\em (i)}, if in addition $\abs{h_1-h_2}<2\abs{z_1-z_2}$, and just an elliptical disk in all other cases. 
 \end{thm}
\begin{proof}For matrices $A$ satisfying either of the conditions (i)--(iii) of Theorem~\ref{th:k=2}, formulas \eqref{mus} for the eigenvalues of $M(\theta)$ given by \eqref{M} take the form:
\begin{multline*} \mu_{1,2} = \frac{1}{2}( h_{1} + h_{2} - 2\Re(z_{11} + z_{22})\cos 2 \theta \\ - 2\im(z_{11} + z_{22})\sin 2 \theta \pm (a\cos 2\theta+b\sin 2\theta+c)),\end{multline*}
where $a,b,c$ solve \eqref{abc}.

The respective ellipses are nested if and only if the difference $\mu_1(\theta)-\mu_2(\theta)$  does not change the sign, in other words, if and only if \eq{1or2} 4\abs{z_1-z_2}^2 = a^2+b^2 \leq c^2.\en

This immediately takes care of case (ii), in which $a=b=0$, and thus \eqref{1or2}  holds. Case (i) is also straightforward. Indeed, in this case $c^2=(h_1-h_2)^2$, so that \eqref{1or2} is equivalent to  $2\abs{z_1-z_2} \leq \abs{h_1-h_2}$.

It remains to consider case (iii). Using \eqref{z121}, the third formula from \eqref{abc} can be rewritten as
\[ c^2=(h_1-h_2)^2+4(\abs{z_{12}}+\abs{z_{21}})^2. \] From here and \eqref{iv}:
\begin{multline*} c^2=(\abs{z_{12}}+\abs{z_{21}})^2\frac{\abs{z_1-z_2}^2+4\abs{{z_{12}}z_{21}}}{\abs{{z_{12}}z_{21}}}\\
\geq (\abs{z_{12}}+\abs{z_{21}})^2\frac{\abs{z_1-z_2}^2}{\abs{{z_{12}}z_{21}}}\geq 4\abs{z_1-z_2}^2, \end{multline*}
implying that \eqref{1or2} does hold. \end{proof}

Note that normal 2-by-2 matrices are essentially Hermitian, and the respective parts of Theorem~\ref{th:1or2} agree with Corollary~\ref{co:tri}. Also, for $Z$ nilpotent part (ii) agrees with Theorem~\ref{th:nilp}.

\providecommand{\bysame}{\leavevmode\hbox to3em{\hrulefill}\thinspace}
\providecommand{\MR}{\relax\ifhmode\unskip\space\fi MR }
\providecommand{\MRhref}[2]{%
  \href{http://www.ams.org/mathscinet-getitem?mr=#1}{#2}
}
\providecommand{\href}[2]{#2}

\end{document}